\pdfoutput=1

\documentclass{amsart}

\usepackage[bookmarks=true,%
    colorlinks=true,%
    linkcolor=blue,%
    citecolor=blue,%
    filecolor=blue,%
    menucolor=blue,%
    urlcolor=blue,%
    breaklinks=true]{hyperref}

\usepackage{bm}
\usepackage{amsthm}
\usepackage{graphicx}

\newcommand\myLangle{\langle~}
\newcommand\myRangle{~\rangle}

\newcommand\extraIEqual{}

\newcommand{\Id}{\mathrm{Id}}

\renewcommand{\H}{\mathbb{H}}
\newcommand{\Q}{\mathbb{Q}}
\newcommand{\Z}{\mathbb{Z}}
\newcommand{\PSL}{\mathrm{PSL}}
\newcommand{\SL}{\mathrm{SL}}

\newtheorem{theorem}{Theorem}[section]
\newtheorem{lemma}[theorem]{Lemma}

\title{All Principal Congruence Link Groups}
\author{M. D. Baker}
\author{M. Goerner}
\author{A. W. Reid}

\address{\newline
IRMAR,\newline
Universit\'e de Rennes 1,\newline
35042 Rennes Cedex,\newline
France.}
\email{mark.baker@univ-rennes1.fr }
\address{\newline 
Pixar Animation Studios, \newline
1200 Park Avenue,\newline
Emeryville, CA 94608, USA.}
\email{enischte@gmail.com}
\address{\newline
Department of Mathematics,\newline
Rice University,\newline
Houston, TX 77005, USA}
\email{alan.reid@rice.edu}

\thanks{Reid was supported in part  by a NSF grant, and The Wolfensohn Fund administered through The Institute for Advanced Study.}

\begin{document}

\begin{abstract}
We enumerate all the principal
congruence link complements in $S^3$, 
thereby answering a question of W\!.~Thurston.
\end{abstract}



\maketitle

%
%
%
%

\section{Introduction}

Let $d$ be a square-free positive integer, let $O_d$ denote the ring
of integers in ${\Q}(\sqrt{-d})$, and $h_d$ denote the class number of ${\Q}(\sqrt{-d})$.

Setting $Q_d={\Bbb H}^3/\PSL(2,O_d)$ to be the Bianchi
orbifold, it is  well-known that $Q_d$ is a finite volume hyperbolic
orbifold with $h_d$ cusps (see \cite{MR} Chapters 8 and 9 for example).
A non-compact finite volume hyperbolic
3-manifold $X$ is called {\em arithmetic} if $X$ and $Q_d$ are
commensurable, that is to say they share a common finite sheeted cover
(see \cite{MR} Chapter 8 for more on this).  

An important class of arithmetic 3-manifolds consists of  the {\em congruence} manifolds.
Recall that a subgroup $\Gamma<\PSL(2,O_d)$ is called a {\em congruence subgroup} if there
exists an ideal $I\subset O_d$ so that $\Gamma$ contains the 
{\em principal congruence group}:
$$\Gamma(I)=\ker\{\PSL(2,O_d)\rightarrow \PSL(2,O_d/I)\},$$

\noindent where $\PSL(2,O_d/I) = \SL(2,O_d/I)/\{\pm \Id\}$.
The largest ideal $I$ for which $\Gamma(I)<\Gamma$ is called the {\em level}
of $\Gamma$. A manifold $M$ is called {\em congruence} (resp. 
{\em principal congruence}) if $M$ is isometric to a manifold ${\Bbb H}^3/\Gamma$ where $\Gamma(I) < \Gamma < \PSL(2,O_d)$ (resp. $\Gamma=\Gamma(I)$) 
for some ideal $I$.  

In an email to the first and third authors in $2009$, W\!.~Thurston asked the following question about principal congruence link complements: \\[\baselineskip]
\noindent {\em  ``Although there are infinitely many arithmetic link complements, there are only finitely many that come from principal congruence subgroups.  Some of the examples known seem to be among the most general (given their volume) for producing lots of exceptional manifolds by Dehn filling, so I'm curious about the complete list."}\\[\baselineskip]
In this note, we enumerate all the principal congruence link complements in $S^3$, together with their levels.  Our main result is the following:

\begin{theorem}
\label{main}
The following list of $48$ pairs $(d,I)$ describes all
principal congruence subgroups $\Gamma(I) < \PSL(2,O_d)$ such that
$\H^3/\Gamma(I)$ is a link complement in $S^3${\normalfont :}

\begin{enumerate}
\item {\bf $d=1$:}~$I=\myLangle 2 \myRangle $, $\extraIEqual\myLangle 2\pm i\myRangle $, $\extraIEqual\myLangle (1\pm i)^3\myRangle $, $\extraIEqual\myLangle 3\myRangle $, $\extraIEqual\myLangle 3\pm i\myRangle $, $\extraIEqual\myLangle 3\pm 2i\myRangle $, $\extraIEqual\myLangle 4\pm i\myRangle $.
\item {\bf $d=2$:}~$I = \myLangle 1\pm \sqrt{-2}\myRangle $, $\extraIEqual\myLangle 2\myRangle $, $\extraIEqual\myLangle 2\pm \sqrt{-2}\myRangle $, $\extraIEqual\myLangle 1\pm 2\sqrt{-2}\myRangle $, $\extraIEqual\myLangle 3\pm \sqrt{-2}\myRangle $.
\item {\bf $d=3$:}~$I = \myLangle 2\myRangle $,  $\extraIEqual \myLangle 3\myRangle $, $\extraIEqual \myLangle (5\pm \sqrt{-3})/2\myRangle $, $\extraIEqual \myLangle 3\pm \sqrt{-3}\myRangle $, $\extraIEqual \myLangle (7\pm \sqrt{-3})/2\myRangle $, $\extraIEqual \myLangle 4\pm \sqrt{-3}\myRangle $, $\extraIEqual \myLangle (9\pm \sqrt{-3})/2\myRangle $.
\item {\bf $d=5$:}~$I=\myLangle 3,(1\pm \sqrt{-5})\myRangle $.
\item {\bf $d=7$:}~$I = \myLangle (1\pm \sqrt{-7})/2\myRangle $, $\extraIEqual \myLangle 2\myRangle $, $\extraIEqual \myLangle (3\pm \sqrt{-7})/2\myRangle $, $\extraIEqual \myLangle \pm \sqrt{-7}\myRangle $, $\extraIEqual \myLangle 1\pm \sqrt{-7}\myRangle $, $\extraIEqual \myLangle (5\pm \sqrt{-7})/2\myRangle $, $\extraIEqual \myLangle 2\pm \sqrt{-7}\myRangle $, $\extraIEqual \myLangle (7\pm \sqrt{-7})/2\myRangle $, $\extraIEqual \myLangle (1\pm 3\sqrt{-7})/2\myRangle $.
\item {\bf $d=11$:}~$I = \myLangle (1\pm \sqrt{-11})/2\myRangle $, $\extraIEqual \myLangle (3\pm \sqrt{-11})/2\myRangle $, $\extraIEqual \myLangle (5\pm \sqrt{-11})/2\myRangle $.
\item {\bf $d=15$:}~$I = \myLangle 2,(1\pm \sqrt{-15})/2\myRangle $, $\extraIEqual\myLangle 3,(3\pm \sqrt{-15})/2\myRangle $, $\extraIEqual\myLangle (1\pm \sqrt{-15})/2\myRangle $, $\extraIEqual\myLangle 5,(5\pm \sqrt{-15})/2\myRangle $, $\extraIEqual\myLangle (3\pm \sqrt{-15})/2\myRangle $.
\item {\bf $d=19$:}~$I =\myLangle (1\pm \sqrt{-19})/2\myRangle $.
\item {\bf $d=23$:}~$I = \myLangle 2,(1\pm \sqrt{-23})/2\myRangle $,  $\extraIEqual \myLangle 3,(1\pm \sqrt{-23})/2\myRangle $, $\extraIEqual \myLangle 4,(3\pm \sqrt{-23})/2\myRangle $.
\item {\bf $d=31$:}~$I = \myLangle 2,(1\pm \sqrt{-31})/2\myRangle $, $\extraIEqual \myLangle 4,(1\pm \sqrt{-31})/2\myRangle $, $\extraIEqual \myLangle 5,(3\pm \sqrt{-31})/2\myRangle $.
\item {\bf $d=47$:}~$I = \myLangle 2,(1\pm \sqrt{-47})/2\myRangle $, $\extraIEqual \myLangle 3,(1\pm \sqrt{-47})/2\myRangle $, $\extraIEqual \myLangle 4,(1\pm \sqrt{-47})/2\myRangle $.
\item {\bf $d=71$:}~$I =\myLangle 2,(1\pm \sqrt{-71})/2\myRangle $.
\end{enumerate}
\end{theorem}

As we will describe in \S{}\ref{sec:prelim} and~\ref{remain}, using previous work of the authors (\cite{BR1}, \cite{BR2}, \cite{Go2}), the proof of Theorem~\ref{main} can be reduced to establishing the following theorem:

\begin{theorem}
\label{main2}
When $d\in\{2,7,11\}$ the following list of pairs $(d,I)$ determine principal congruence subgroups $\Gamma(I) < \PSL(2,O_d)$ such that
${\Bbb H}^3/\Gamma(I)$ is a link complement in $S^3${\normalfont :}

\begin{enumerate}
\item {\bf $d=2$:}~$I=\myLangle 1\pm 2\sqrt{-2}\myRangle $, $\extraIEqual\myLangle 3\pm \sqrt{-2}\myRangle $.
\item {\bf $d=7$:}~$I = \myLangle \pm \sqrt{-7}\myRangle $, $\extraIEqual \myLangle (5\pm \sqrt{-7})/2\myRangle $, $\extraIEqual \myLangle 2\pm \sqrt{-7}\myRangle $, $\extraIEqual \myLangle (7\pm \sqrt{-7})/2\myRangle $, $\extraIEqual \myLangle (1\pm 3\sqrt{-7})/2\myRangle $.
\item {\bf $d=11$:}~$I = \myLangle (5\pm \sqrt{-11})/2\myRangle $.
\end{enumerate}

\noindent Furthermore, in the case when $d=2$, $\Gamma(\myLangle 1+3\sqrt{-2}\myRangle )$ is not a link group.\end{theorem}

Theorem~\ref{main2} will be deduced by combining previous work  of the authors, as well as further applications of these techniques,  together with Lemma~\ref{lemma:nasty}, which deals with the  case 
$(2, \langle 1+3\sqrt{-2} \rangle)$. In contrast to the other cases, this final case required finding an automatic structure for a certain group for which we used the program Monoid Automata Factory (MAF) \cite{MAF}.

We finish the introduction with some commentary. Rather than a collaboration, this paper (and the associated technical report \cite{BGR3}) is the conclusion of overlapping efforts of the first and third authors and independently the second author.  It was suggested to the authors by Ian Agol that since Theorem~\ref{main} was proved
almost simultaneously,  a collaborative effort should be undertaken to describe the final solution. The main goal of this note is to provide a brief overview of previous work and summary of techniques that lead to Theorem~\ref{main2}.  

The proof of Theorem~\ref{main2} is largely computational and builds upon (for the most part) the techniques developed in our previous independent work to deal with earlier cases. 
This work work relied heavily on {\sc Magma} \cite{Mag}, GAP \cite{Gap} and SnapPy \cite{SnapPy}. In the light of this, we have decided to 
present the work here essentially in ``announcement form'', deferring the technical details 
including the {\sc Magma} routines and SnapPy computations of homology for congruence manifolds to the technical report \cite{BGR2} (for example see \S 14 of \cite{BGR2} which summarizes the {\sc Magma} calculations).  In particular, a complete proof of Theorem \ref{main2} (and Theorem \ref{main})  can be obtained in two ways from \cite{BGR2}: by combining Section 3 of \cite{BGR2} with either Part 1 (the work of the second author) or Part 2 (the work of the first and third authors) of \cite{BGR2}.  In addition, in a companion preprint \cite{BGR3} we describe several new principal congruence link diagrams, which we intend to update when more link diagrams for the remaining principal congruence manifolds are identified. 

We refer the reader to \cite{BR3} for further background, history and connections with other questions regarding the topology of congruence 
link complements.\\[\baselineskip]

\noindent{\bf Acknowledgements:}~{\em The work of the first and third author was developed over  multiple visits to the University of Texas and the Universit\'e de Rennes 1.  They also wish to thank the Universit\'e Paul Sabatier,
the Max Planck Institut, Bonn, I.C.T.P. Trieste and The Institute for Advanced Study for their support and  hospitality as this work unfolded over several years.
In addition they also wish to thank several people without whose help we would not have been able to complete this work: M. D. E. Conder, D. Holt, E. A. O'Brien, A. Page, M. H. Sengun and A. Williams. The second author wishes to
thank R.C. Haraway III, N. Hoffman and M. Trnkova for helpful discussions about Dirichlet domains and M. Culler and N. Dunfield for explaining how SnapPy computes homology groups for large triangulations. All of the authors also wish to thank I. Agol for his suggestion to pool our resources on this paper, and the referee for many helpful suggestions.}

\section{Preliminaries and techniques} \label{sec:prelim}

In this section, we review some earlier work that was used in \cite{BR1}, \cite{BR2}, and \cite{Go2} to produce a finite list of potential pairs $(d,I)$.

Note that if $I\subset O_d$ is an ideal and $\overline{I}\subset O_d$ the complex conjugate ideal, then $\Gamma(I)$ is a link group if and only if $\Gamma(\overline{I})$ is a link group --- since
complex conjugation induces an orientation-reversing involution of ${\Bbb H}^3/\Gamma(I)$. Hence it suffices to consider only one of the ideals $I$ and $\overline{I}$ as a candidate level for a link group.

\subsection{Reducing to finitely many cases} \label{sec:whyFinCases}

Suppose that $L\subset S^3$ is a link with $n$ components, and $M\cong S^3\setminus L$. Abusing notation and identifying $M$ with the link exterior, set $\iota: \partial M\rightarrow M$ to be the inclusion map. Now
$H_1(M;\Z)\cong \Z^n$ and $H_1(M;\Z)/\iota_* (H_1(\partial M;\Z))$ is trivial (i.e.~link complements have trivial cuspidal cohomology). Hence the solution to the 
Cuspidal Cohomology Problem (completed in \cite{Vo}) provides the following consequence for principal congruence link complements:

\begin{theorem}
\label{cusp}
Suppose that $M=\H^3/\Gamma(I)$ is homeomorphic to a link complement and $M\rightarrow Q_d$. Then $M$, and hence $Q_d$,  has trivial cuspidal cohomology, and so
$$d \in \{1, 2, 3, 5, 6, 7, 11, 15, 19, 23, 31, 39, 47, 71\}.$$
\end{theorem}

First, note that
if  $\Gamma(I)$ is a link group, it must be a torsion-free subgroup of $\PSL(2,O_d)$. We can disregard 
the case when $I=O_d$, since the groups $\PSL(2,O_d)$ all contain elements of orders $2$ and $3$. It can then be easily checked that 
there are only $6$ pairs $(d,I)$, up to complex conjugation with $d$ as above, so that $\Gamma(I)$ contains a non-trivial element of finite order.

To pass from finitely many values of $d$ to finitely many possible pairs $(d,I)$ the norm of the ideal $I$ needs to be bounded. To achieve this, we follow the arguments of \cite{BR1} and \cite{BR2}.  We note in passing that a different combinatorial method was used in \cite{Go2} in his analysis of the cases of $d=1,3$.
When the class number is $1$, we can use the $6$-Theorem of Agol \cite{Ag} and Lackenby \cite{Lac}
to control which peripheral curves can produce $S^3$ by Dehn filling. 
When the class number is greater than $1$, an upper bound for the systole for a hyperbolic link complement in $S^3$ from \cite{AR} can be used. Since systole length grows with the norm of the ideal this provides the necessary control.
Moreover, when the class number is $1$, all ideals are principal and the argument above bounds the absolute value of a generator of the ideal, and when the class number is greater than $1$, this bounds
the absolute value of {\em some} $x\in I$.  Summarizing this discussion we obtain (see \cite[Section 4.1]{BR1} and \cite[Lemma~4.1]{BR2}):

\begin{theorem}
\label{finitenumber}
If $(d, I)$ determines a link complement, then there must be a non-trivial $x\in I$ with $|x| < 6$ (if $h_d=1$), respectively, $|x|^2 < 39$ (if $h_d > 1$). In particular there are only finitely many pairs $(d,I)$.
\end{theorem}

Theorem \ref{finitenumber} reduces the classification of which principal congruence groups are link groups 
to a finite list, indeed there are $302$ cases (up to complex conjugation and excluding the $6$ cases that give groups containing elements of finite order).

\subsection{Proving a case is not a link complement} \label{section:notLink}

Assume that $(d,I)$ is one of the finitely many pairs provided by Theorems \ref{cusp} and \ref{finitenumber} for which we need to decide that $M=\H^3/\Gamma(I)$ is, or is not, homeomorphic to a link complement.
We first discuss the case of proving that $M$ is not a link complement. Indeed, we will show something slightly stronger, namely that $\Gamma$ is not generated by parabolic elements (equivalently its peripheral subgroups).

To describe how this is achieved, fix a collection of $\PSL(2,O_d)$-inequivalent cusps $c_i$ for $i=1,\ldots , h_d$, let $P_i$ be the peripheral subgroup of  $\PSL(2,O_d)$ fixing the cusp $c_i$, set
$P_i(I)=P_i\cap \Gamma(I)$ to be the peripheral subgroup of $\Gamma(I)$ fixing $c_i$, and
let $N_d(I)$ denote the normal closure in $\PSL(2,O_d)$ of $\{P_1(I),\ldots ,P_{h_d}(I)\}$. Note that $N_d(I) < \Gamma(I)$ since $\Gamma(I)$ is a normal subgroup of $\PSL(2,O_d)$. 
Both $M$ and $\H^3/N_d(I)$ are covering spaces of the Bianchi orbifold $Q_d$ with the covering groups given by $\PSL(2,O_d/I)$, and $\PSL(2,O_d)/N_d(I)$ respectively. Now $\Gamma(I)$ will be generated by parabolic elements 
if and only if $|\PSL(2,O_d)/N_d(I)|=|\PSL(2,O_d/I)|$. Since $|\PSL(2,O_d/I)|$ can easily be computed from a factorization of $I$ (see \cite[Section~2.1]{BR2}), this reduces the problem to determining $|\PSL(2,O_d)/N_d(I)|$.

For many small values of $d$, finite presentations of the group $\PSL(2,O_d)$ together with the matrices corresponding to the generators were computed by Swan \cite{Sw}. More recently, for the remaining values of $d$,
Page \cite{Page} computed such presentations (see \cite{BR2}).  If we add the words representing the generators of each $P_j(I)$ 
as relations to the finite presentation of $\PSL(2, O_d)$, we have a finite presentation for $\PSL(2,O_d)/N_d(I)$. From this finite presentation, we can use various techniques from computer algebra to compute (a lower bound for) the size of $\PSL(2,O_d)/N_d(I)$. 

Often, however, it is sufficient (but not always easier) to compute or estimate the homology of $M$ itself to prove that it is not a link complement. In fact, this suffices for all but the three cases $(1,4+3\sqrt{-1})$, $(2, 1+3\sqrt{-2})$, and $(3,(11+\sqrt{-3})/2)$ where $H_1(M;\Bbb Z)/\iota_*(H_1(\partial M;\Bbb Z))$ is trivial.
More recently, the second author wrote a computer program to compute and triangulate a Dirichlet domain for a Bianchi orbifold $Q_d$ and construct covers of $Q_d$ to generate a triangulation of the principal congruence manifold $M$, respectively, the congruence manifold associated to the upper unit-triangular matrices in $\PSL(2,O_d/I)$. Using this program and the fact that $H_1(M;\Bbb Z)/\iota_*(H_1(\partial M;\Bbb Z))$ cannot vanish for a cover $M\to N$ of degree  less than $|H_1(N;\Bbb Z)/\iota_*(H_1(\partial N;\Bbb Z))|$, it is feasible to compute integral homology for enough congruence manifolds to rule out all but the aforementioned three cases. Further discussion of this program and the output of these computations is contained in \cite{BGR2}.

\subsection{Proving a case is a link complement}

By Perelman's resolution of the Geometrization Conjecture, to prove that $M={\Bbb H}^3/\Gamma(I)$ is homeomorphic to a link complement in $S^3$, it is sufficient to find Dehn fillings of the manifold $M$ trivializing the fundamental group.  Thus the task is to find a collection of slopes (essential simple closed curves), one from each cusp, so that killing these words trivializes the fundamental group.
The first and third author did this by exhibiting a set of parabolic elements, one from each cusp
subgroup of N(I), that generates Gamma(I).
The second author used the computer program mentioned above to generate a triangulation of $M$ and then used SnapPy to find the slopes (using techniques similar to those already described in \cite[Section~7.3.2]{Go2}) 
for which SnapPy could then prove that the fundamental group of the Dehn-filled manifold along those slopes is trivial.  Further details are in \cite{BGR2}.

\section{The remaining cases}
\label{remain}
We now review how our previous work using the methods of \S{}\ref{sec:prelim} reduces the classification of principal congruence link groups to the cases in Theorem~\ref{main2}. 

In the case of $h_d > 1$,  the complete list of the $16$ pairs $(d,I)$ 
corresponding to principal congruence link complements was determined in \cite{BR2}.  The possible values of $d$ are $5,15,23,31,47,71$ with the levels shown in Theorem~\ref{main}. 

Concerning the case when $h_d = 1$ (i.e. $d\in\{1,2,3,7,11,19\}$), certain examples already existed in the literature (see \cite{BR1}), and using the $6$-Theorem as described in \S \ref{sec:prelim} to restrict the possible
levels, we subsequently gave $9$ new examples of principal congruence link groups in \cite {BR1}. This brought the total known when $h_d=1$ to $18$:

\begin{enumerate}
\item {\bf $d=1$:}~$I=\myLangle 2\myRangle$, $\extraIEqual \myLangle 2\pm i\myRangle$, $\extraIEqual\myLangle (1\pm i)^3\myRangle$, $\extraIEqual\myLangle 3\myRangle$.
\item {\bf $d=2$:}~$I = \myLangle 1\pm \sqrt{-2}\myRangle $, $\extraIEqual\myLangle 2\myRangle $, $\extraIEqual\myLangle 2\pm 2\sqrt{-2}\myRangle $.
\item {\bf $d=3$:}~$I = \myLangle 2\myRangle $,  $\extraIEqual \myLangle 3\myRangle $, $\extraIEqual \myLangle (5\pm \sqrt{-3})/2\myRangle $, $\extraIEqual \myLangle 3\pm \sqrt{-3}\myRangle $.
\item {\bf $d=7$:}~$I = \myLangle (1\pm \sqrt{-7})/2\myRangle $, $\extraIEqual \myLangle 2\myRangle $, $\extraIEqual \myLangle (3\pm \sqrt{-7})/2\myRangle $, $\extraIEqual \myLangle 1+\pm \sqrt{-7}\myRangle $. 
\item {\bf $d=11$:}~$I = \myLangle (1\pm \sqrt{-11})/2\myRangle $, $\extraIEqual \myLangle (3\pm \sqrt{-11})/2\myRangle $.
\item {\bf $d=19$:}~$I =\myLangle (1\pm \sqrt{-19})/2\myRangle $.
\end{enumerate}

Moreover, in the cases $d = 1,3$,  as well as identifying the cases described above, in \cite {Go2} the second author determined the complete list of pairs $(d,I)$ that yield link groups; namely those above, together with:

\begin{enumerate}
\item {\bf $d=1$:}~$I=\myLangle 3\pm i\myRangle $, $\extraIEqual\myLangle 3\pm 2i\myRangle $, $\extraIEqual\myLangle 4\pm i\myRangle $.
\item {\bf $d=3$:}~$I = \myLangle (7\pm \sqrt{-3})/2\myRangle $, $\extraIEqual \myLangle 4\pm \sqrt{-3}\myRangle $, $\extraIEqual \myLangle (9\pm \sqrt{-3})/2\myRangle $.
\end{enumerate}

The upshot of these combined works is that  $40$ pairs $(d,I)$ were determined that yield principal congruence link groups, and using a combination of the techniques described in \S \ref{sec:prelim}, all remaining cases were eliminated except for the $8$ pairs $(d,I)$ stated in Theorem~\ref{main2}, and $(2,\langle 1+3\sqrt{-2} \rangle)$.

In Table~\ref{table:MoreLinkComplements} we provide some additional information associated to the $8$ cases to be identified as link groups: 
in the second, third, and fourth columns of  Table~\ref{table:MoreLinkComplements}, we list a generator $x$ of
the ideal being considered, its 
norm and the order of $\PSL(2,O_d)/\Gamma(I)$.

\begin{table}[h]
\caption{The 8 cases where we must prove that $\Gamma(\myLangle x\myRangle)$ is a link group
\label{table:MoreLinkComplements}}
\begin{center}
\begin{tabular}{|c|c|c|c|c|}
\hline
$d$&$x$&$N(\myLangle x\myRangle)$&$|\PSL(2,O_d/\myLangle x\myRangle)|$&Number of cusps\\
\hline
$2$&$1+2\sqrt{-2}$&$9$&$324$& $36$\\
$2$&$3+\sqrt{-2}$&$11$&$660$&$60$\\
\hline
$7$&$\sqrt{-7}$&$7$&$168$&$24$\\
$7$&$(5+\sqrt{-7})/2$&$8$&$192$&$24$\\
$7$&$2+\sqrt{-7}$&$11$&$660$&$60$\\
$7$&$(7+\sqrt{-7})/2$&$14$&$1008$&$72$\\
$7$&$(1+3\sqrt{-7})/2$&$16$&$1536$&$96$\\
\hline
$11$&$(5+\sqrt{-11})/2$&$9$&$324$&$36$\\
\hline
\end{tabular}
\end{center}
\end{table}

Details of the computations establishing that these do indeed give link groups are provided in \cite{BGR2}.

\section{The final case}

As mentioned in the introduction, the final case required a technique different from the other cases to prove the finitely presented group $G=\PSL(2,O_d)/N_d(I)$ to be large enough. For the other cases, this could be shown in {\sc Magma} or GAP  either by computing $|G|$ itself or the abelianization of $G$ or a suitable subgroup of $G$. The necessary methods for this are based on Todd-Coxeter coset enumeration, Reidemeister-Schreier rewriting, and Smith Normal form. However, the final case was solved using the Monoid Automata Factory (MAF). We are very grateful to Alun Williams who helped us with this.

\begin{lemma} \label{lemma:nasty}
The principal congruence manifold ${\Bbb H}^3/\Gamma(\langle 1+3\sqrt{-2} \rangle)$ is not homeomorphic to a link complement in $S^3$.
\end{lemma}

\begin{proof}
From \cite{Sw}, we have the following presentation for 
$$\PSL(2,O_2)=\langle a,t,u | a^2=(ta)^3=(au^{-1}au)^2=tut^{-1}u^{-1}=1\rangle$$
where $$a=\begin{pmatrix} 0 & 1 \\ -1 & 0 \end{pmatrix},\quad t = \begin{pmatrix}1 & 1\\ 0 & 1\end{pmatrix},\quad \mbox{ and} \quad u = \begin{pmatrix}1 & \sqrt{-2}\\ 0 & 1\end{pmatrix}.$$
Following \S{}\ref{section:notLink}, we obtain $t^6u^{-1}$ and $t^{19}$ as the two parabolic elements that normally generate $N_2(\langle 1+3\sqrt{-2}\rangle)$, giving us the following presentation
$$G=\frac{\PSL(2,O_2)}{N_2(\langle 1+3\sqrt{-2}\rangle)} = \langle a,t,u | a^2=(ta)^3=(au^{-1}au)^2=tut^{-1}u^{-1}=t^6u^{-1}=t^{19}=1\rangle.$$
We give this presentation of $G$ to MAF \cite{MAF} which proved that $G$ is infinite (see \cite[Lemma 3.2]{BGR2} for how this was coded). From the discussion in \S{}\ref{section:notLink}, we deduce that $\Gamma(\langle 1+3\sqrt{-2}\rangle)$ is not a link group. 
\end{proof}

\end{document}